\documentclass[letterpaper,10pt]{amsart}
\usepackage{fullpage,amsfonts,amsmath,amssymb, amsthm, graphicx, subfig}
\usepackage{indentfirst}
\usepackage{tikz}
\usepackage{hyperref}
\usepackage[numbers]{natbib}
\usepackage{listings}
\definecolor{egrn}{RGB}{61,142,76}
\definecolor{epurp}{RGB}{188,38,240}
\definecolor{ecomment}{RGB}{179,38,44}
\definecolor{eteal}{RGB}{42,142,155}

\usetikzlibrary{arrows}
\newcommand{\Mac}{{\ttfamily Macaulay2}}

\newcommand{\SP}{{\ttfamily SimplicialPosets}}
\newcommand{\rk}{\operatorname{rank} }

\theoremstyle{definition}
\newtheorem{mydef}{Definition}
\newtheorem{Def1}[mydef]{Definition}

\newtheorem{mydef2}[mydef]{Definition}
\newtheorem{mydef3}[mydef]{Definition}
\newtheorem{separation}[mydef]{Definition}

\theoremstyle{plain}
\newtheorem{prop}{Proposition}
\newtheorem{theorem}[prop]{Theorem}
\newtheorem{theorem5}[prop]{Theorem}

\theoremstyle{plain}

\theoremstyle{definition}
\newtheorem{example}{Proposition}
\newtheorem{eg1}[example]{Example}
\newtheorem{eg2}[example]{Example}
\newtheorem{eg3}[example]{Example}
\newtheorem{eg4}[example]{Example}

\newcommand{\comment}[1]{}

\begin{document}

\title{A Macaulay2 Package for Simplicial Posets}
\author{\textsc{Nathan Nichols}\\\textsc{\small University of Minnesota (Twin Cities)}}
\date{\textsc{\small Monday, July 6, 2020}}
\begin{abstract}
       We give a description of a new \Mac\ package called \SP. This package provides functions for working with simplicial posets and calculating their generalized Stanley--Reisner ideals. For practical purposes, we also introduce of a new random model for a class of simplicial posets which generalizes existing models for random simplicial complexes such as the Kahle model.
\end{abstract}

\maketitle

\section{Introduction}

A simplicial poset $P$ is a poset with a unique minimum element $\hat{0}$ such that for any element $v\in P$, the lower set $L_v = [\hat{0},v]$ is a boolean lattice. Simplicial posets were originally introduced by Stanley in \cite{MR1117642}, who defined an associated ring $A_P$ that generalizes the Stanley--Reisner ring of a simplicial complex. In his original paper, Stanley shows that certain properties of simplicial posets are easier to prove than their analogs for simplicial complexes.

The primary purpose of the package \textit{SimplicialPosets} is to add functionality for working with simplicial posets and their associated generalized Stanley--Reisner rings to Macaulay2. This expands the existing capability in Macaulay2 of defining simplicial complexes and computing their Stanley--Reisner rings. Section \ref{basicOps} is devoted to describing the core features of this package. 

The purpose of Section \ref{mathSection} is to mathematically define a deterministic function $\Theta(\Delta_1, \Delta_2)$ that takes two simplicial complexes and produces a simplicial poset. This provides a way of defining a class of \textit{weak simplicial posets} to which existing models of random simplicial complexes can be generalized by setting $\Delta_1$ and $\Delta_2$ to random variables. A more precise description of the properties that define these weak simplicial posets is proven in Theorem \ref{thm5}. For readers who are interested in learning more about the topic of random simplicial complexes, the work of De Loera et al \cite{MR3906176} and the survey by Kahle \cite{MR3290093} are recommended.

\section{Basic operations}\label{basicOps}
A basic capability of this package is to check whether or not a given poset is a simplicial poset. This is done using the function {\ttfamily isSimplicial}:
\begin{lstlisting}
i2 : B := booleanLattice 4;
i3 : isSimplicial B
o3 = true
\end{lstlisting}
Given a poset that is assumed to be simplicial, the function {\ttfamily isFacePoset} determines whether or not it is isomorphic to the face poset of some simplicial complex. For example:
\begin{lstlisting}
i1 : gndR = QQ[a,b,c,d,e];
i2 : P = facePoset simplicialComplex({a*b*c,b*c,a*e});
i3 : isFacePoset P
o3 = true
\end{lstlisting}
An important capability of this package is calculating the generalized Stanley--Reisner ideal associated with a simplicial poset as defined by Stanley in \citep{MR1117642}. Here is an example demonstrating that the Stanley--Reisner ideal of a simplicial complex is equal to the Stanley simplicial poset ideal of its face poset:
\begin{lstlisting}
i1 : gndR = QQ[a,b,c,d];
i2 : C = simplicialComplex({a*b*c,b*c*d});
i3 : P = facePoset C;
i4 : I1 = minimalPresentation stanleyPosetIdeal P;
i5 : I2 = ideal(C);
i6 : M = map(ring I2, ring I1, vars ring I2);
i7 : I2 == M(I1)
o7 = true
\end{lstlisting}

These previous three examples do not make use of any simplicial posets that are not isomorphic to the face poset of some simplicial complex. Defining more complicated simplicial posets is a task that requires some more thorough explanation.

Let $A,B$ be two simplicial posets such that $A\cap B = \hat{0}$ and let $X\subseteq A$ and $Y\subseteq B$ be order ideals. If $X\cong Y$ by an isomorphism $\sigma:X\to Y$, $A$ and $B$ can be \textit{glued} together by forming a quotient poset $(X\cup Y) / \backsim$ where $\backsim$ is the relation that equates $v\in X$ with $\sigma(v)\in Y$ while leaving all other elements of $X\cup Y$ in singleton equivalence classes. The equivalence classes of this relation are ordered by the rule $[v]\leq [w]$ if and only if there exists $v'\in [v]$  and $w'\in [w]$ such that  $v' \leq w'$.

In this software package, the operation {\ttfamily deltaGlue} glues two simplicial posets along isomorphic order ideal $X,Y$. To fully specify the isomorphism between the order ideals $X\subseteq A$ and $Y\subseteq B$, we require a hash table that associates the atoms of $X$ with the atoms of $Y$ and another hash table that sends the maximal elements of $X$ to the maximal elements of $Y$. Here is an example: 
\begin{lstlisting}
i1 : A = booleanLattice 4;
i2 : B = naturalLabeling booleanLattice 4;
i3 : HT1 = new HashTable from {"1110" => 11, "0111" => 14};
i4 : HT2 = new HashTable from {"1000" => 1, "0100" => 2, "0010" => 3, "0001" => 4};
i5 : P = deltaGlue(A, B, HT1, HT2);
\end{lstlisting}

The resulting poset $P$ has two maximal elements that correspond to two distinct 3-dimensional simplices on the same 4 points. A caveat of the function {\ttfamily deltaGlue} is that it requires the assignment of facets to be consistent with the assignment of atoms.  For example, consider the following commands:
\begin{lstlisting}
i1 : A = booleanLattice 4;
i2 : B = naturalLabeling booleanLattice 4;
i3 : HT1 = new HashTable from {"1110" => 11, "0111" => 14};
i4 : HT2 = new HashTable from {"1000" => 4, "0100" => 2, "0010" => 3, "0001" => 1};
i5 : P = deltaGlue(A,B, HT1, HT2);
error: Assignment of atoms-atoms or facets-facets invalid.
\end{lstlisting}

The reason for this error is that the function {\ttfamily deltaGlue} expects the atoms below a maximal element {\ttfamily x} of $X$ to be mapped to the atoms below the element {\ttfamily HT1\#x} of $Y$. Here, the atoms below {\ttfamily "0111"} in {\ttfamily A} are {\ttfamily \{"0100", "0010", "0001"\}} and the atoms below {\ttfamily HT1\#"0111" = 14} in {\ttfamily B} are {\ttfamily \{4, 2, 3\}}, but {\ttfamily \{"0100", "0010", "0001"\}} does not map to {\ttfamily\{4, 2, 3\}} under the hash table {\ttfamily HT2}. Namely, {\ttfamily HT2\#"0001"} is {\ttfamily 1}.

\section{The $\Theta$ gluing function}\label{mathSection}

Recall that simplicial posets arise as the posets of cells of a $\Delta$-complex ordered by inclusion. Topological constructions on $\Delta$-complexes such as taking a quotient space can be thought of in terms of the combinatorial operations that they induce on the $\Delta$-complex's simplicial poset of cells. In this section, we will frequently refer to the operation of taking a quotient space. It will be more convenient to work directly with the combinatorial definitions instead of their topological counterparts.

This construction is restated from a paper due to Bj\"{o}rner \cite{MR746039} but was originally defined by Garsia and Stanton \cite{MR736732}. Here, $L_a$ denotes the lower set of the poset element $a$.
\begin{mydef}[\cite{MR746039}, Section 2.3]\label{gluing}
Let $P$ be the face poset of a simplicial complex $\Delta$. An equivalence relation $\backsim$ on P is called a \textit{gluing relation} if it has the following properties:

\begin{enumerate} 
  \item $\tau\backsim\sigma$ and $\sigma \neq \tau$ implies that $\tau$ and $\sigma$ are incomparable with $\rk(\tau)=\rk(\sigma)$. Also, $\tau$ and $\sigma$ must have no common upper bound.
  \item If $\tau\backsim\sigma$, then every element of $L_{\tau}$ must also be related (under the relation $\backsim$) to some element of $L_{\sigma}$.
\end{enumerate}

A poset $P/\backsim$ is defined as the set of equivalence classes of $\backsim$ ordered by the following rule:

$$C_1 \leq C_2 \iff \exists v\in C_1 \textrm{ s.t. } v \leq w \textrm{ for some }w\in C_2$$

\end{mydef}

In \cite{MR746039}, Bj\"{o}rner states that all simplicial posets (which are referred to as ``posets of boolean type") can be constructed as a quotient by a gluing relation on some simplicial complex. More precisely stated, it is known that for any simplicial poset $P$ there exists a simplicial complex $\Delta$ and a gluing relation $\approx$ on $\Delta$ such that $P \cong \Delta/\approx$. 

In general, the simplicial complex $\Delta$ and the gluing relation $\approx$ are not unique. For the purposes of defining the function $\Theta$, it will be informative to prove this fact by constructing a canonical simplicial complex and a canonical gluing relation from which an arbitrary simplicial poset $P$ can be realized as a quotient. 

\begin{separation}\label{defseparation}
Let $Q$ be an arbitrary simplicial poset. Let $F$ be the set of maximal vertices of $Q$. Define a simplicial poset $P$ called the {\textit{separation}} of $Q$ as follows:

$$P = [\bigsqcup\limits_{x\in F}(L_x-\hat{0})]\cup\hat{0}$$

Two elements $(i,v),(j,w)\in P$ are comparable if and only if $i=j$. If $(i,v)$ is comparable to $(j,w)$, $(i,v)\leq (j,w)$ if and only if  $v\leq w$. The reason why the element $\hat{0}$ is subtracted from each term in the disjoint union is to ensure that the minimum element of $P$ is unique. 
\end{separation}

Note that the separation of a simplicial poset is always the face poset of a simplicial complex.

\begin{theorem}\label{gluingThm} Let $P'$ be the separation of simplicial poset $P$. Then, $P$ can be realized (up to isomorphism) as a gluing of $P'$. 
\end{theorem}
\begin{proof}
Vertices of $v'\in P'$ are of the form $v' = (i, v)$ for some $i\in\mathbb{N}$ and $v\in P$. Let $\sigma$ be the function defined by $(i, v)\mapsto v$. Next, define a relation $\backsim$ on $P'$:
$$v \backsim w \iff \sigma(v)= \sigma(w)$$
It is straightforward to verify that $\backsim$ is a gluing relation and $\sigma$ is an isomorphism from $P'/\backsim$ to $P$.
\end{proof}

Let $\Delta_1$ and $\Delta_2$ be simplicial complexes on a set of points $G$. In this paper, a ``simplicial complex on the points $G$" is assumed to actually contain all the points in $G$. The reason why we adopt this non-standard convention is because the function $\Theta(\Delta_1, \Delta_2)$ will require $\Delta_1$ and $\Delta_2$ to have the \textit{the same} sets of points. If $\Delta_1$ does not contain some point $p\in G$, it can be considered as a simplicial complex on the set of points $G':= G\setminus \{p\}$ instead.

Let $P_1$ and $P_2$ be the face posets of $\Delta_1$ and $\Delta_2$ respectively. Let $P_1'$ be the separation of $P_1$ and $\sigma:P_1'\to P_1$ be defined as above. We define another relation $\doteq$ on $P_1'$ as follows:
$$a\doteq b \iff \sigma(a)=\sigma(b)\textrm{ and }\sigma(a)\in P_1 \cap P_2$$

It is relatively easy to check that the relation $\doteq$ is a gluing relation. The relation $\doteq$ also has the property that its equivalence classes refine the partition induced by the relation $\backsim$. 

\begin{mydef2}\label{thetaDef}
For simplicial complexes $\Delta_1,\Delta_2$ on the same set of points, 
$$\Theta(\Delta_1, \Delta_2) := P_1'/\doteq$$
\end{mydef2}

At this point, the properties of the function $\Theta$ (namely, what simplicial posets it may be used to construct) are not obvious. In Theorem \ref{thm5}, we will prove a necessary and sufficient condition for a given simplicial poset to be of the form $\Theta(\Delta_1, \Delta_2)$ for some simplicial complexes $\Delta_1, \Delta_2$. First, it may be informative to look at an example of how $\Theta$ can be used to construct a simplicial poset which is not the face poset of a simplicial complex:

\begin{eg3}\label{ex3}
Let $\Delta$ be any simplicial complex on the set of points $[n]=\{1,\ldots, n\}$ and let $x,y$ be unique symbols not contained in $[n]$. Define $\hat{\Delta}$ to be the simplicial complex on the points $G=[n]\cup \{x,y\}$ determined by the following two properties:

\begin{enumerate}
\item[1.] If $f\subseteq [n]$ is a facet of $\Delta$, then $f$ is also a facet of $\hat{\Delta}$.
\item[2.] The sets $\{x\}$ and $\{y\}$ are facets of $\hat{\Delta}$.
\end{enumerate}

Let $\Omega$ be the simplicial complex on the points $G$ whose facets are precisely $[n]\cup \{x\}$ and $[n]\cup \{y\}$. The poset $P=\Theta(\Omega,\hat{\Delta})$ has two maximal elements which correspond to the two facets of $\Omega$. The lower sets of these two maximal elements have an intersection isomorphic to the face poset of the simplicial complex $\hat{\Delta}$. Unless $\Delta$ is a simplex on $[n]$, $P$ is not the face poset of a simplicial complex because the intersection of the lower sets of its maximal elements is not a boolean algebra.
\end{eg3}
The following definitions will be useful for proving Theorem \ref{thm5}:
\begin{Def1}\label{atomFamDef}
Let $P$ be a simplicial poset. For a maximal element $v\in P$, let $A(v)$ denote the set of atoms of $P$ below $v$. Define $A(P)$ as the following subset family:
$$A(P):=\{A(v) : v \textrm{ maximal in P}\}$$

$A(P)$ will be called the \textit{atom family} of $P$.
\end{Def1}

\begin{mydef3}\label{meetDef}
Let $P$ be a simplicial poset with set $F$ of maximal elements. Define the \textit{meet poset} $M(P)$ as follows:
$$M(P) =\bigcup_{\substack{ A,B\in F \\ A\neq B}}(L_A\cap L_B)$$
\end{mydef3}
\begin{eg4}
The poset $P=\Theta(\Omega, \hat{\Delta})$ defined in Example \ref{ex3} has the property that $M(P)\cong \Delta$.
\end{eg4}
We now proceed to prove the main theorem of this section.

\begin{theorem5}\label{thm5}
Let $P$ be a simplicial poset. Consider the following propositions:
\begin{enumerate} 
  \item[(i)] As a subset family of the atoms of P, $A(P)$ is an antichain (i.e., there is no element of A(P) that contains another element of A(P).)
  \item[(ii)] $M(P)$ is isomorphic to the face poset of some simplicial complex.
\end{enumerate}
     Claim: $P$ is equal to $\Theta(\Delta_1,\Delta_2)$ (up to isomorphism) for some simplicial complexes $\Delta_1,\Delta_2$ if and only if (i) and (ii) hold.
\end{theorem5}
\begin{proof}

Assume $P \cong \Theta(\Delta_1,\Delta_2) = (P'_1/\doteq)$ for simplicial complexes $\Delta_1$ and $\Delta_2$. It must be shown that conditions (i) and (ii) hold:

\begin{enumerate}
\item[i):] Let $f:P_1' \to P_1'/\doteq$ and $g:P_1' \to P_1'/\backsim$ be the canonical quotient maps. By Definition \ref{thetaDef}, $\Delta_1$ and $\Delta_2$ must include the same set of points. In particular, all atoms of $P_1$ are always contained in the intersection $P_1 \cap P_2$. An implication of this is that $\doteq$ and $\backsim$ are equivalent relations when restricted to the set of atoms of $P_1'$. This implies that for all atoms $a\in P_1'$, $f(a) = g(a)$.

The lower sets of maximal elements of $P_1'$ map surjectively to the lower sets of maximal elements under the quotient maps $f,g$. Thus, $A(P_1'/\doteq)$ and $A(P_1'/\backsim)$ can be written in terms of $A(P_1')$ as follows:
$$A(P_1'/\doteq) = \biggr\{\{f(x)\ |\ x\in X\}\ \biggr|\ X\in A(P_1')\biggr\}$$
$$A(P_1'/\backsim) = \biggr\{\{g(x)\ |\ x\in X\}\ \biggr|\ X\in A(P_1')\biggr\}$$
\noindent Because $f(a)=g(a)$ for all atoms $a\in P_1'$, we may conclude that: 
$$A(P_1'/\doteq) = A(P_1'/\backsim)$$
Recall from the proof of Theorem \ref{gluingThm} that $P_1'/\backsim$ is isomorphic to the face poset of the simplicial complex $\Delta_1$. Since it is well-known that the face poset of any simplicial complex satisfies property (i), this proves that $A(P)$ also satisfies property (i).
\bigskip
\item[ii):] By Theorem \ref{gluingThm}, $(P_1'/\backsim) \cong P_1$ where $P_1$ is the face poset of $\Delta_1$. By taking the meet poset of both sides, we obtain:
$$M(P_1'/\backsim) \cong M(P_1)$$
\noindent It is relatively easy to see that in the case of $\doteq$, the above equation becomes:  
$$M(P_1'/\doteq) \cong M(P_1) \cap P_2$$

Formally, the face poset of an (abstract) simplicial complex $\Delta$ is the same set as $\Delta$ except considered to have an ordering by inclusion. An abstract simplicial complex is the union of a set of simplices on the same set of points, and the intersection of any two simplices is always a simplex. Thus, $M(P_1)$ is the face poset of a simplicial complex that we may denote  $M(\Delta_1)$ and $M(P_1) \cap P_2$ is the face poset of the simplicial complex $M(\Delta_1) \cap \Delta_2$.
\end{enumerate}
\medskip
\noindent For the other direction, assume that (i) and (ii) hold. We will construct simplicial complexes $\Delta_1$ and $\Delta_2$ such that $P\cong \Theta(\Delta_1, \Delta_2)$:
\medskip

Because $P$ satisfies property (i), there exists a unique simplicial complex $\Delta_1$ whose facets are the sets in $A(P)$. Since $P$ satisfies property (ii), there exists a unique simplicial complex $\Delta_2=M(P_1)\cup G$ where $G$ is the set of points of $\Delta_1$. This choice of $\Delta_1$ and $\Delta_2$ is such that: 
\begin{equation}
\begin{split}
P_1 \cap P_2 &= P_1 \cap (M(P_1)\cup G)\\
&= P_1 \cap M(P_1) = M(P_1)\\
\end{split}
\end{equation}
It is relatively easy to see that this implies that $\doteq$ and $\backsim$ are equivalent relations. This implies the following equation:
$$(P_1'/\doteq) = (P_1'/\backsim) = \Theta(\Delta_1,\Delta_2)$$
\noindent By the proof of Theorem \ref{gluingThm}, there is a canonical isomorphism $(P_1'/\backsim)\to P$.

\end{proof}

Here are two examples that show properties (i) and (ii) are independent of each other:

\begin{eg1}\label{ex4}
The following is the Hasse diagram of a simplicial poset that satisfies (ii) but not (i):
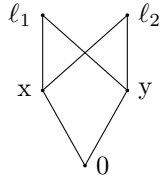
\begin{figure}[h]
\centering
\begin{tikzpicture}[scale=0.75, vertices/.style={draw, fill=black, circle, inner sep=0pt}]
             \node [vertices, label=right:{0}] (0) at (-0+0,0){};
             \node [vertices, label=left:{x}] (1) at (-.75+0,1.33333){};
             \node [vertices, label=right:{y}] (2) at (-.75+1.5,1.33333){};
             \node [vertices, label=left:{$\ell_1$}] (4) at (-.75+0,2.66667){};
             \node [vertices, label=right:{$\ell_2$}] (3) at (-.75+1.5,2.66667){};
     \foreach \to/\from in {0/1, 0/2, 1/4, 1/3, 2/4, 2/3}
     \draw [-] (\to)--(\from);
     \end{tikzpicture}
\end{figure}
\end{eg1}   

\begin{eg2}\label{example4} The following $\Delta$-complex has a face poset $P$ such that $M(P)$ is the simplicial poset from Example \ref{ex4}.

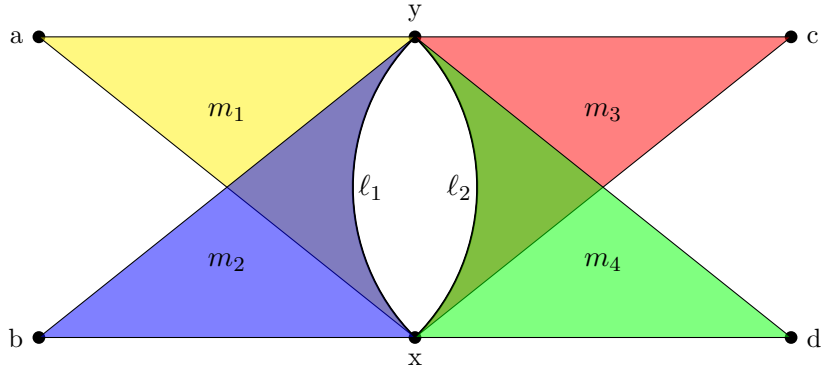
\begin{figure}[h]
\centering
\def\sep{1/2}

\def\baryx{0.3}
\def\baryy{0.3}
\usetikzlibrary{patterns}
\def\sep{5}
\def\baryx{0.5}
\def\baryy{0.5}
\def\angAA{135}
\def\angAB{225}
\def\angBA{45}
\def\angBB{315}
\begin{tikzpicture}
    \tikzstyle{point}=[circle,thick,draw=black,fill=black,inner sep=0pt,minimum width=4pt,minimum height=4pt]
    \node (x)[point,label=below:x] at (0,-2) {};
    \node (y)[point,label=above:y] at (0,2) {};
    \node (a)[point,label=left:{a}] at (-\sep,2) {};
    \node (b)[point,label=left:{b}] at (-\sep,-2) {};
    \node (c)[point,label=right:{c}] at (\sep,2) {};
    \node (d)[point,label=right:{d}] at (\sep,-2) {};

    \draw [thick] (x.center) to[out=135,in=225] (y.center);
    \draw [thick] (y.center) to[out=\angBB,in=\angBA] (x.center);

	\draw[fill=red, fill opacity=0.5] (c.center) -- (x.center) to[out=\angBA,in=\angBB]  (y.center) -- cycle;
	\draw[fill=yellow, fill opacity=0.5] (a.center) -- (y.center) to[out=\angAB,in=\angAA]  (x.center) -- cycle;
	\draw[fill=blue, fill opacity=0.5] (b.center) -- (x.center) to[out=\angAA,in=\angAB]  (y.center) -- cycle;
	\draw[fill=green, fill opacity=0.5] (d.center) -- (y.center) to[out=\angBB,in=\angBA]  (x.center) -- cycle;

	\node at (barycentric cs:a=1,x=\baryx,y=\baryy) {\large$m_1$};
	\node at (barycentric cs:b=1,x=\baryx ,y=\baryy) {\large$m_2$};
	\node at (barycentric cs:c=1,x=\baryx,y=\baryy) {\large$m_3$};
	\node at (barycentric cs:d=1,x=\baryx ,y=\baryy) {\large$m_4$};
    \node (p1)[,label=right:{\large$\ell_1$}] at (-1,0) {};
    \node (p2)[label=left:{\large$\ell_2$}] at (1,0) {};
\end{tikzpicture}
\end{figure}
\noindent The atom family of this $\Delta$-complex is $\{\{a,x,y\}, \{b,x,y\}, \{c,x,y\}, \{d,x,y\}\}$. Hence, its face poset satisfies (i) but not (ii).
\end{eg2}

\section{Random simplicial posets}\label{moreexamples}

In the package \SP, the function $\Theta$ can be computed using the function {\ttfamily thetaGlue}. Here is an example of how {\ttfamily thetaGlue} can be used to define a simplicial poset that is not the face poset of a simplicial complex:

\begin{lstlisting}
i1 : gndR = QQ[a,b,c,x,y];
i2 : A = simplicialComplex({a*b*c*x, a*b*c*y});
i3 : B = simplicialComplex({a*b,b*c,a*c});
i4 : P = thetaGlue(A,B);
i5 : isFacePoset P
o5 = false
\end{lstlisting}

In this example, the meet poset of the face poset of $A$ is $|\ abc\ |$. The resulting simplicial poset $P$ is not a face poset of a simplicial complex because $|\ abc\ | \cap B = |\ ab\ bc\ ac\ |\neq |\ abc\ |$.

The function {\ttfamily randSimplicialPoset(n, p1, p2)} is {\ttfamily thetaGlue(A,B)} in the special case that {\ttfamily A} and {\ttfamily B} are both random simplicial posets generated using the Kahle model as  defined in \cite{MR2510573}. The argument {\ttfamily n} is the number of points and {\ttfamily p1}, {\ttfamily p2} are Erd\H{o}s--R\'enyi model probability parameters for the simplicial complexes {\ttfamily A}, {\ttfamily B} respectively. Here is an example of how {\ttfamily randSimplicialPoset} can be used to perform an experiment involving random simplicial posets:
\begin{lstlisting}
i1 : D = 1..100;
i2 : L = apply(D, i-> if isFacePoset randSimplicialPoset(6,0.5,0.5) then 1 else 0);
i3 : tally L
o3 = Tally{0 => 40}
           1 => 60
\end{lstlisting}
The function {\ttfamily meetPoset} returns $M(P)$ as defined in Definition \ref{meetDef}. Also, the function {\ttfamily atomFamily} returns the atom family $A(P)$ as defined in Definition \ref{atomFamDef}. Together with the function {\ttfamily isAntichainList}, it is possible to verify that a given simplicial poset satisfies the conditions (i) and (ii) of Theorem \ref{thm5}:
\begin{lstlisting}
i1 : P = randSimplicialPoset(5,0.75,0.75);
i2 : isAntichainList atomFamily P
o2 = true
i3 : isFacePoset meetPoset P
o3 = true
\end{lstlisting}
\section{Acknowledgements}
 Thank you to Victor Reiner for being an invaluable mentor and Mahrud Sayrafi for encouraging me to undertake this project. I am also highly grateful to the reviewers for providing insightful feedback.

\medskip

\bibliographystyle{plainnat}
\bibliography{main}

\end{document}